\documentclass{amsproc}
\usepackage{amssymb}
\usepackage{amsmath}
\newtheorem{theorem}{Theorem}[section]
\newtheorem{lemma}[theorem]{Lemma}
\newtheorem{proposition}[theorem]{Proposition}

\newcommand{\RR}{\mathbb{R}}

\newcommand{\NN}{\mathbb{N}}
\newcommand{\BofX}{\mathcal{B}(\XX)}
\newcommand{\KofX}{\mathcal{K}(\XX)}
\newcommand{\XX}{\mathfrak{X}}

\newcommand{\Mat}{\mathrm{Mat}}

\begin{document}
\title[Generalized Berger-Wang formula]{The generalized Berger-Wang formula and the spectral radius of linear cocycles}
\author{Ian D. Morris}
\begin{abstract}

Using multiplicative ergodic theory we prove two formulae describing the relationships between different joint spectral radii for sets of bounded linear operators acting on a Banach space. In particular we recover a formula recently proved by V. S. Shulman and Yu. V. Turovski\u{\i} using operator-theoretic ideas. As a byproduct of our method we answer a question of J. E. Cohen on the limiting behaviour of the spectral radius of a measurable matrix cocycle.

\end{abstract}
\maketitle
\section{Introduction and statement of results}
Let $\mathsf{A}$ be a set of $d\times d$ real matrices. The joint spectral radius of $\mathsf{A}$ was defined by G.-C. Rota and W. G. Strang in \cite{RS2} to be the quantity
\[\hat\varrho(\mathsf{A}):= \lim_{n \to \infty} \sup\left\{\|A_{i_n}\cdots A_{i_1}\|^{1/n} \colon A_i \in \mathsf{A}\right\}.\]
The joint spectral radius has since emerged as a useful tool in a number of research areas including the theory of control and stability \cite{Ba,Gu,Koz}, coding theory \cite{MO}, wavelet regularity \cite{DL,DL0,Mae2}, and the study of  numerical solutions to ordinary differential equations \cite{GZ}. The following characterisation of the joint spectral radius, which we term the \emph{Berger-Wang formula}, is due to M. A. Berger and Y. Wang \cite{BW}, following a conjecture in \cite{DL}:
\begin{theorem}\label{BuggerWank}
Let $\mathsf{A}$ be a nonempty bounded set of $d \times d$ real matrices. Then
\[\hat\varrho(\mathsf{A})=\limsup_{n \to \infty} \sup \left\{\rho(A_{i_n}\cdots A_{i_1})^{1/n}\colon A_i \in \mathsf{A}\right\}.\]
\end{theorem}

The concept of joint spectral radius generalizes directly to the context of bounded operators on Banach spaces, where it may be used to establish a number of results in invariant subspace theory \cite{S,ST1,ST3}. In this article we use ergodic theory to prove an analogue of Theorem \ref{BuggerWank} for the case in which $\mathsf{A}$ is a set of bounded linear operators acting on a Banach space $\XX$. We also obtain a further relationship between two other joint spectral radii which may be defined in this context. In order to state our main results we require the following definitions.

Let $\XX$ be a Banach space and let $B_{\XX}$ denote its unit ball. We define the \emph{Hausdorff measure of noncompactness} of an operator $L \in \BofX$, which we denote by $\|L\|_\chi$, to be the infimum of all positive real numbers $\varepsilon$ for which $LB_{\XX}$ admits a finite $\varepsilon$-net. We define a second functional on $\BofX$ by $\|L\|_f:=\inf \{\|L-K\|\colon \mathrm{rank}\,K<\infty\}$. Some useful properties of these two quantities are described in the following proposition. 
\begin{proposition}\label{basic}
Let $\XX$ be a Banach space. Then the functions $\|\cdot\|_\chi, \|\cdot\|_f \colon \BofX \to \RR$ are seminorms and are Lipschitz continuous. For every $L \in \BofX$ we have $\|L\|_\chi = 0 \iff L \in \KofX$ and $\|L\|_\chi \leq \|L\|_f \leq \|L\|$. The seminorms $\|\cdot\|_\chi$ and $\|\cdot\|_f$ are \emph{submultiplicative:} for every $L_1,L_2 \in \BofX$ we have $\|L_1L_2\|_\chi \leq \|L_1\|_\chi \|L_2\|_\chi$ and similarly for $\|\cdot\|_f$.
\end{proposition}
The properties listed above for $\|\cdot\|_\chi$ are rather standard (see e.g. \cite{Nutter}); the properties of $\|\cdot\|_f$ and the inequality $\|L\|_\chi \leq \|L\|_f$ may easily be deduced by the reader. 

Given a Banach space $\XX$ and a bounded set $\mathsf{A} \subseteq \BofX$, let us define $\mathsf{A}^n := \{A_{i_n}\cdots A_{i_1} \colon A_i \in 
\mathsf{A}\}$ for each $n \in\NN$. We consider the following four spectral radii:
\[\hat\varrho(\mathsf{A}) := \lim_{n \to \infty} \sup_{A \in \mathsf{A}^n}   \|A\|^{1/n},\qquad\varrho_\chi(\mathsf{A}) := \lim_{n \to \infty} \sup_{A \in \mathsf{A}^n}   \|A\|_\chi^{1/n},\]
\[\varrho_f(\mathsf{A}) := \lim_{n \to \infty} \sup_{A \in \mathsf{A}^n}   \|A\|_f^{1/n},\qquad \varrho_r(\mathsf{A}) := \limsup_{n \to \infty} \sup_{A \in \mathsf{A}^n} \rho(A)^{1/n}.\]
By submultiplicativity it follows that the limit in each of the definitions of $\hat\varrho$, $\varrho_\chi$ and $\varrho_f$ exists and is also equal to the infimum over $n \in \mathbb{N}$ of the same quantity.

In this article we establish the following basic relationships between the spectral radii $\hat\varrho$, $\varrho_r$, $\varrho_\chi$ and $\varrho_f$:
\begin{theorem}\label{Boq}
Let $\XX$ be a Banach space and let $\mathsf{A} \subseteq \BofX$ be precompact and nonempty. Then
\begin{equation}\label{GBWF}
\hat\varrho(\mathsf{A})=\max\{\varrho_\chi(\mathsf{A}),\varrho_r(\mathsf{A})\}
\end{equation}
and
\begin{equation}\label{HE}
\varrho_\chi(\mathsf{A})=\varrho_f(\mathsf{A}).
\end{equation}
\end{theorem}
The equation \eqref{GBWF} has been termed the \emph{Generalized Berger-Wang formula} by V. S. Shulman and Yu. V. Turovski\u{\i} \cite{ST1}, who gave a proof conditional on various additional hypotheses. A more general version was established in \cite{ST2}, with an unconditional proof recently being given in \cite{ST3}. The relation \eqref{HE} does not appear to have been known prior to the present article, although a similar relation is known in cases where $\XX$ satisfies a strong version of the compact approximation property \cite[p.419]{ST1}.

As well as strictly generalising Theorem \ref{BuggerWank}, Theorem \ref{Boq} yields simple proofs of several results on invariant subspaces of families of operators. In particular it may be applied to give a short proof of a celebrated theorem of Turovski\u{\i} \cite{Tu} which states that any semigroup of compact quasinilpotent operators acting on a Banach space has a nontrivial invariant subspace. Proofs via \eqref{GBWF} of this and other related results may be found in \cite{ST1,ST3}. 

It should be noted that \eqref{GBWF} may fail to hold if we assume only that  $\mathsf{A}$ is bounded, and also that there exist Banach spaces $\XX$ and precompact sets $\mathsf{A} \subseteq \BofX$ such that $\varrho_r(\mathsf{A})<\varrho_\chi(\mathsf{A})$. Specifically, in the case where $\XX$ is a separable infinite-dimensional Hilbert space, M.-H. Shih \emph{et al.} give an example in \cite{SWP} of a closed bounded noncompact set $\mathsf{A} \subseteq \BofX$ which satisfies $\varrho_\chi(\mathsf{A})=\varrho_r(\mathsf{A})=0$ and $\hat\varrho(\mathsf{A})=1$. Also in the context of a separable Hilbert space, an example was given by P. Rosenthal and A. So{\l}tysiak in \cite{RS} of a two-element set $\mathsf{A} \subseteq \mathcal{B}(\XX)$ such that $\varrho_r(\mathsf{A})<\hat\varrho(\mathsf{A})$. It is not clear to the author whether \eqref{HE} can fail to hold when $\mathsf{A}$ is bounded but not precompact.

We now describe the ergodic-theoretic results which are used to deduce Theorem \ref{Boq}. Let $T$ be a measure-preserving transformation of a probability space $(X, \mathcal{F},\mu)$ and $\XX$ a Banach space. A \emph{linear cocycle} over $T$ is a measurable function $\mathcal{A} \colon X \times \mathbb{N} \to \mathcal{B}(\XX)$ which satisfies the relation
\[\mathcal{A}(x,n+m) = \mathcal{A}(T^mx,n)\mathcal{A}(x,m)\]
for every $n, m \in \mathbb{N}$ and $\mu$-almost-every $x \in X$. If $X$ is a topological space and $\mathcal{F}$ its Borel $\sigma$-algebra we say that $\mathcal{A}$ is a continuous linear cocycle if the map $x \mapsto \mathcal{A}(x,k)$ is continuous for every $k \in \mathbb{N}$. When discussing linear cocycles we shall find it useful to adopt the conventions $\log 0 := -\infty$ and $\log^+ (x) :=\max\{0,\log x\}$ for $x \geq 0$.

Theorem \ref{Boq} is derived from the following technical result:
\begin{theorem}\label{Boq2}
Let $T \colon X \to X$ be a homeomorphism of a compact metric space, $\mu$ a $T$-invariant Borel probability measure on $X$, and $\XX$ a Banach space. Suppose that $\mathcal{A} \colon X \times \mathbb{N} \to \mathcal{B}(\XX)$ is a continuous cocycle such that $\mathcal{A}(x,n)$ is an injective operator for every $(x,n) \in X \times \mathbb{N}$. Then there exists a $T$-invariant Borel set $\Lambda \subseteq X$ satisfying $\mu(\Lambda)=1$ with the following properties. If $x \in \Lambda$ then the limits
\[\lambda(x) := \lim_{n \to \infty} \frac{1}{n} \log \|\mathcal{A}(x,n)\|,\]
\[\chi(x) := \lim_{n \to \infty} \frac{1}{n} \log \|\mathcal{A}(x,n)\|_\chi\]
exist,
\begin{equation}\label{anzab}
\lim_{n \to \infty} \frac{1}{n}\log \|\mathcal{A}(x,n)\|_f = \chi(x),
\end{equation}
and if additionally $\chi(x)<\lambda(x)$ then
\begin{equation}\label{bwanza}
\limsup_{n \to \infty} \frac{1}{n} \log \rho(\mathcal{A}(x,n)) = \lambda(x).\end{equation}
\end{theorem}
For $d \in \NN$ let $\Mat_d(\RR)$ denote the vector space of all $d \times d$ real matrices. The ideas used to prove Theorem \ref{Boq2} may also be applied to obtain the following result which answers a question of J. E. Cohen \cite[p.329]{C}:\begin{theorem}\label{Boq3}
Let $T$ be a measure-preserving transformation of a probability space $(X,\mathcal{F},\mu)$ and let $\mathcal{A} \colon X \times \mathbb{N} \to \Mat_d(\mathbb{R})$ be a measurable linear cocycle such that $\int\log^+ \|\mathcal{A}(x,1)\|\,d\mu(x)<\infty$. Let $Z$ denote the set of all $x \in X$ such that $\lim_{n \to \infty} \|\mathcal{A}(x,n)\|^{1/n}$ exists. Then
\begin{equation}\label{bwan}
\mu\left(\left\{x \in Z \colon \limsup_{n \to \infty} \frac{1}{n} \log \rho(\mathcal{A}(x,n)) = \lim_{n \to \infty} \frac{1}{n} \log\|\mathcal{A}(x,n)\|
\right\}\right)=1.\end{equation}
\end{theorem}
A version of Theorem \ref{Boq3} was given by A. Avila and J. Bochi \cite{AB} in the case where $T$ is invertible and $\mathcal{A}(x,1) \in SL_2(\RR)$ for every $x \in X$. Avila and Bochi also give an example to show that $\lim_{n \to \infty} \frac{1}{n} \log\rho(\mathcal{A}(x,n))$ can fail to exist $\mu$-a.e. even when the dependence between $\mathcal{A}(x,1)$ and $\mathcal{A}(T^kx,1)$ is very weak.
\section{Proof of Theorems \ref{Boq2} and \ref{Boq3}}

The main ingredient of the proof of Theorem \ref{Boq2} is a multiplicative ergodic theorem for Banach spaces which was established by P. Thieullen \cite{T}, building on earlier work of R. Ma\~n\'e \cite{Mane} and D. Ruelle \cite{Ruelle2}. In order to state Thieullen's result we require the following definitions.

Given a Banach space $\XX$ we let $\mathcal{G}(\XX)$ denote the set of all closed subspaces $F \subseteq \XX$ for which there exists a closed subspace $G \subseteq \XX$ such that $\XX = F \oplus G$. In particular, $\mathcal{G}(\XX)$ contains all finite-dimensional subspaces of $\XX$ as well as all closed subspaces of finite codimension. Given any $F_0 \in \mathcal{G}(\XX)$ and any $G \in \mathcal{G}(\XX)$ such that $F_0 \oplus G = \XX$, define $U_{F_0,G}=\{ F \in \mathcal{G}(\XX) \colon  F \oplus G =\XX\}$. Define a map $\varphi_{F_0,G} \colon U_{F_0,G} \to \mathcal{B}(F_0,G)$ by identifying each $F \in U_{F_0,G}$ with the unique continuous linear map $F_0 \to G$ whose graph is $F$; that is, $\varphi_{F_0,G}(F)$ is the restriction to $F_0$ of the unique projection $\XX\to G$ having image $G$ and kernel $F$, this projection being continuous by the closed graph theorem. We now define a topology on $\mathcal{G}(\XX)$ by declaring each triple $(U_{F_0,G}, \varphi_{F_0,G},\mathcal{B}(F_0,G))$ to be a chart at $F_0$. The resulting topology has the property that $F_n \to F$ in $\mathcal{G}(\XX)$ if and only if $\varphi_{F,G}(F_n) \to 0 \in \mathcal{B}(F,G)$ for every $G \in \mathcal{G}(\XX)$ such that $F \oplus G = \XX$; see \cite[Appendix B]{T}.

Given topological spaces $\Omega_1, \Omega_2$ and a Borel measure $\mu$ on $\Omega_1$, we shall say that a map $f \colon \Omega_1 \to \Omega_2$ is \emph{$\mu$-continuous} if there exists a sequence of pairwise disjoint compact sets $K_n \subseteq X$ such that $\mu(\bigcup_n K_n)=1$ and the restriction of $f$ to $K_n$ is continuous for each $n$.

We now give a statement of Thieullen's theorem, restricted to the special case of a compact metric space $X$ and ergodic measure $\mu$. To simplify the statement let us write $\NN_p :=\{1,\ldots,p\}$ for $p \in \NN$ and $\NN_\infty:=\NN$. 
\begin{theorem}\label{Theelun}
Let $T \colon X \to X$ be a continuous homeomorphism of a compact metric space, let $\mu$ be an ergodic Borel probability measure on $X$, let $\XX$ be a Banach space, and let $\mathcal{A} \colon X \times\NN \to \BofX$ be a continuous cocycle. Suppose that $\mathcal{A}(x,n)$ is an injective operator for every $(x,n )\in X\times\NN$. By the subadditive ergodic theorem the limits $\chi := \lim \frac{1}{n} \log \|\mathcal{A}(x,n)\|_\chi$ and $\lambda: = \lim \frac{1}{n}\log \|\mathcal{A}(x,n)\|$ exist and are constant $\mu$-a.e. Suppose that $\chi<\lambda$. Then there exists a $T$-invariant Borel set $\Lambda \subseteq X$ which satisifies $\mu(\Lambda)=1$ such that the following properties hold.

There exist $p \in \NN \cup \{\infty\}$, a sequence of real numbers $\lambda=\lambda_1>\lambda_2>\ldots> \chi$ indexed in $\NN_p$, and two corresponding sequences of Borel-measurable $\mu$-continuous maps $F_1,F_2,\ldots \colon \Lambda \to \mathcal{G}(\XX)$,  $G_1,G_2,\ldots, \colon \Lambda \to \mathcal{G}(\XX)$ indexed in $\NN_p$, such that for every $x \in \Lambda$ and $k \in \NN_p$, $F_k(x)$ is finite-dimensional and
\[F_1(x) \oplus \cdots \oplus F_k(x) \oplus G_k(x) = \XX,\]
\[\mathcal{A}(x,n)F_k(x)=F_k(T^nx),\]
\[\mathcal{A}(x,n)G_k(x) = G_k(T^nx)\]
for every $n \in \NN$. For each $x \in \Lambda$ and $k \in \NN_p$ we have
\[\lim_{n \to \infty} \frac{1}{n}\log \frac{\|\mathcal{A}(x,n)v\|}{\|v\|} = \lambda_k\]
uniformly for $v \in F_k(x) \setminus \{0\}$. Similarly, if $k \in \NN$ with $k<p$ then for each $x \in \Lambda$,
\begin{equation}\label{erglehek}\lim_{n \to \infty} \frac{1}{n}\log \sup\left\{\frac{\|\mathcal{A}(x,n)v\|}{\|v\|} \colon  v \in G_k(x)\setminus\{0\}\right\} = \lambda_{k+1}.\end{equation}
If $k=p<\infty$ then the limit in \eqref{erglehek} is instead equal to $\chi$. If $p=\infty$ then $\lim_{i \to \infty}\lambda_i=\chi$.\end{theorem}
\emph{Remark}. The statement of uniform convergence in Theorem \ref{Theelun} is not announced explicitly in the statement of that theorem in Thieullen's paper. However, the corresponding statement is proved in \cite[p.68-69]{T}.

We may now begin the proof of Theorem \ref{Boq2}. Let $X, T, \XX$ and $\mathcal{A}$ be as in that theorem. We must show that the set of all $x \in X$ such that
\[ \limsup_{n \to \infty} \rho(\mathcal{A}(x,n))^{1/n} = \liminf_{n \to \infty} \|\mathcal{A}(x,n)\|^{1/n}= \limsup_{n \to \infty} \|\mathcal{A}(x,n)\|^{1/n}\]
and
\[\liminf_{n \to \infty} \|\mathcal{A}(x,n)\|^{1/n}_\chi= \limsup_{n \to \infty} \|\mathcal{A}(x,n)\|^{1/n}_f,\]
which is clearly Borel, has full measure for every $T$-invariant measure $\mu$. Using a suitable ergodic decomposition theorem it suffices to prove that this holds in cases where $\mu$ is ergodic, which assumption we make for the remainder of the section.

We begin by showing that $\lim \frac{1}{n}\log \|\mathcal{A}(x,n)\|_f = \lim \frac{1}{n}\log \|\mathcal{A}(x,n)\|_\chi$  $\mu$-a.e. Let $\lambda$, $\chi$ be as above; if $\lambda = \chi$ then the result follows directly from Proposition \ref{basic}, so we assume $\lambda>\chi$. Let $\Lambda \subseteq X$, $F_1,F_2,\ldots$ and $G_1,G_2,\ldots$ be as given by Theorem \ref{Theelun}. Given $\varepsilon>0$, take $k \geq 1$ such that for every $x \in \Lambda$ the limit in \eqref{erglehek} is bounded by $\chi+\varepsilon$. For each $x \in \Lambda$ define $F(x) := F_1(x) \oplus \cdots \oplus F_k(x)$ and let $P(x) \in \mathcal{B}(\XX)$ be the unique projection having image $F(x)$ and kernel $G_k(x)$. Note that for each $x\in\Lambda$ the boundedness of $P(x)$ is guaranteed by the closed graph theorem, and that the $\mu$-continuity of the maps $F_1,\ldots,F_k$, $G_k$ implies the $\mu$-continuity of $P$. Now take any $x \in \Lambda$. Since $P(x)$ has finite rank,
\begin{align*}
\|\mathcal{A}(x,n)\|_f &\leq \|\mathcal{A}(x,n) - \mathcal{A}(x,n)P(x)\|\\ &\leq (1+\|P(x)\|) \sup \left\{\frac{\|\mathcal{A}(x,n)v\|}{\|v\|} \colon v \in G_k(x) \setminus \{0\} \right\}\end{align*}
so that $\limsup \frac{1}{n} \log \|\mathcal{A}(x,n)\|_f \leq \chi+\varepsilon$ for each $x \in \Lambda$. By Proposition \ref{basic} we have $\liminf \frac{1}{n} \log \|\mathcal{A}(x,n)\|_f \geq \chi $ and the result follows.

We now proceed to the second part of Theorem \ref{Boq2}. Our approach is suggested by recent work of B. Kalinin \cite{K}; this line of argument is also applied by the author in \cite{M}. For each $x \in \Lambda$ we take $V(x) = F_1(x)$ and $W(x) = G_1(x)$, let $P(x) \in \BofX$ be the unique projection having image $V(x)$ and kernel $W(x)$, and define $Q(x) = I - P(x)$. Clearly $P$ and $Q$ are well-defined and $\mu$-continuous as before. Let $x \in  \Lambda$ and $n \in \mathbb{N}$; if $v \in V(x)$ and $w \in W(x)$ then clearly
\[P(T^nx)\mathcal{A}(x,n)(v+w) = \mathcal{A}(x,n)v = \mathcal{A}(x,n)P(x)(v+w),\]
and since $V(x)\oplus W(x) = \XX$ it follows that $P(T^nx)\mathcal{A}(x,n) = \mathcal{A}(x,n)P(x)$ and $Q(T^nx)\mathcal{A}(x,n) = \mathcal{A}(x,n)Q(x)$. We require the following lemma:
\begin{lemma}\label{lemmer}
There exists a set $\tilde \Lambda \subseteq \Lambda$ with $T\tilde\Lambda \subseteq \tilde\Lambda$ and $\mu(\tilde\Lambda)=1$ such that for every $x \in \tilde \Lambda$,\[\liminf_{n \to \infty} \|P(x) - P(T^nx)\| =0.\]
\end{lemma}
\begin{proof}
It suffices to show that for each $r>0$, the set
\[\Lambda_r := \left\{x \in \Lambda \colon \liminf_{n \to \infty} \|P(x) - P(T^nx)\| \leq 2/r\right\}\]
has full measure, since we may then define
\[\tilde\Lambda = \bigcap_{n=0}^\infty T^{-n}\left( \bigcap_{r=1}^\infty \Lambda_r\right)\]
and obtain the desired result. 

Let $K_n$ be a sequence of compact subsets of $X$ witnessing the $\mu$-continuity of the map $x \mapsto P(x)$. Define
\[\mathcal{Z} = \bigcup_{n \geq 1} \{P(x) \colon x \in K_n\} \subseteq \BofX.\]
Clearly $\mathcal{Z}$ is a countable union of compact sets, hence separable,
and $\mu(\{x \in \Lambda \colon P(x) \in \mathcal{Z}\})=1$. Let $r>0$. Since $\mathcal{Z}$ is separable we may choose a sequence $(L_n)_{n \geq 1}$ in $\mathcal{Z}$ such that $\{B_{1/r}(L_n) \colon n \geq 1\}$ covers $\mathcal{Z}$. For each $n>0$ let $C_{n,r} = \{x \in \Lambda \colon P(x) \in B_{1/r}(L_n)\}$. For each $n>0$ the Poincare recurrence theorem yields
\[\mu(\{x \in C_{n,r} \colon T^k x \in C_{n,r} \text{ for infinitely many }k \in \NN
\}) = \mu(C_{n,r})\]
and since $\mu(\bigcup_{n \geq 1} C_{n,r})=1$ it follows that $\mu(\Lambda_r)=1$ as required.\end{proof}
For each $\delta>0$ and $x \in \Lambda$ define a closed convex cone in $\XX$ by
\[K(x,\delta) := \left\{u \in \XX \colon \|P(x)u\| \geq \delta^{-1}\|Q(x) u\|\right\}.\]
To prove Theorem \ref{Boq2} it suffices to show that each $x \in \tilde \Lambda$ has the following two properties: firstly, for every sufficiently small $\varepsilon>0$ we have
\begin{equation}\label{pfgoal}\inf_{u \in K(x,1) \setminus \{0\}}\frac{\|\mathcal{A}(x,n)u\|}{\|u\|} \geq e^{n(\lambda-3\varepsilon)}\end{equation}
for all sufficiently large $n>0$; and secondly, for infinitely many $n>0$ we have $\mathcal{A}(x,n)K(x,1) \subseteq K(x,1)$. To see that this implies \eqref{bwanza}, note that if $\mathcal{A}(x,n)K(x,1)$ is contained in  $ K(x,1)$ and \eqref{pfgoal} holds, then taking any $v \in K(x,1) \setminus \{0\}$ we obtain
\[\frac{1}{n}\log \rho(\mathcal{A}(x,n)) \geq \liminf_{k \to \infty} \frac{1}{nk} \log \left\|(\mathcal{A}(x,n))^kv\right\| \geq \lambda-3\varepsilon.\]
Given $x \in \tilde\Lambda$ it follows that \eqref{bwanza} will be satisfied if the above conditions can be met for every $\varepsilon>0$. We therefore fix $x \in \tilde \Lambda$ for the remainder of the proof and proceed to establish these two properties. By Theorem \ref{Theelun} we have
\[\liminf_{n \to \infty} \frac{1}{n}\log \inf \left\{\frac{\|\mathcal{A}(x,n)v\|}{\|v\|} \colon v \in V(x) \setminus \{0\}\right\} =\lambda\]
and
\[\limsup_{n \to \infty} \frac{1}{n}\log \sup \left\{\frac{\|\mathcal{A}(x,n)v\|}{\|v\|} \colon v \in W(x) \setminus \{0\}\right\} = \nu\]
for some $\nu<\lambda$.  Choose any $\varepsilon>0 $ small enough that $3\varepsilon<\lambda - \nu$. If $n$ is taken large enough we have for each $u \in K(x,1)$
\begin{align*}\|P(T^nx)\mathcal{A}(x,n)u\|=\|\mathcal{A}(x,n)P(x)u\| &\geq e^{n(\lambda-\varepsilon)}\|P(x)u\|\\
&\geq \frac{1}{2}e^{n(\lambda-\varepsilon)}\|u\| \geq e^{n(\lambda-2\varepsilon)}\|u\|\end{align*}
and
\[\|Q(T^nx)\mathcal{A}(x,n)u\| = \|\mathcal{A}(x,n)Q(x)u\| \leq e^{n(\nu+\varepsilon)}\|Q(x)u\| \leq e^{n(\nu+\varepsilon)} \|Q(x)\|.\|u\|\]
where we have used the inequality $\|u\|=  \|(P(x)+Q(x))u \| \leq 2\|P(x)u\|$ which holds for all $u \in K(x,1)$. Combining the above expressions yields
\[\|Q(T^nx)\mathcal{A}(x,n)u\| \leq e^{n(\nu + 3\varepsilon - \lambda)}\|Q(x)\|.\|P(T^nx)\mathcal{A}(x,n)u\|\]
for every $u \in K(x,1)$, from which we conclude that for each $\delta>0$ we have $\mathcal{A}(x,n) K(x,1) \subseteq K(T^nx,\delta)$ for all large enough $n$. Additionally we obtain
\begin{align*}\|\mathcal{A}(x,n)u\| &\geq \|P(T^nx)\mathcal{A}(x,n)u\| - \|Q(T^nx)\mathcal{A}(x,n)u\|\\
 &\geq \left(e^{n(\lambda-2\varepsilon)} - e^{n(\nu+\varepsilon)}\|Q(x)\|\right)\|u\|\end{align*}
for every $u \in K(x,1)$, which gives \eqref{pfgoal} when $n$ is large enough.

To complete the proof we show that for every $\delta \in (0,1)$ we have $K(T^nx,\delta) \subseteq K(x,1)$ for infinitely many $n$. Given $\delta \in (0,1)$, choose $\kappa >0$ such that $\delta^{-1} - 2\kappa(1+\delta^{-1})>1$. By Lemma \ref{lemmer} we have $\|P(x)-P(T^nx)\|<\kappa$ for infinitely many $n>0$. For each such $n$ we have
\begin{align*}
\|P(x)u\| \geq \|P(T^nx)u\| - \kappa\|u\| &\geq \delta^{-1}\|Q(T^nx)u\| - \kappa \|u\|\\
&\geq \delta^{-1} \|Q(x) u\|- \kappa(1+\delta^{-1})\|u\|\end{align*}
for every $u \in K(T^nx,\delta)$, where we have used the relation $\|P(x)- P(T^nx)\| = \|Q(x) - Q(T^nx)\|$ which follows from the definition of $Q$. If $u \in K(T^nx,\delta) \setminus K(x,1)$ then additionally $\|u\| < 2\|Q(x)u\|$ and therefore
\[\|P(x) u \| > \left(\delta^{-1} - 2\kappa(1+\delta^{-1})\right)\|Q(x)u\| \geq \|Q(x)u\|\]
contradicting $u \notin K(x,1)$. We conclude that $K(T^nx,\delta) \setminus K(x,1) = \emptyset$ and therefore $K(T^nx,\delta)\subseteq K(x,1)$ as required. The proof of Theorem \ref{Boq2} is complete.

The proof of Theorem \ref{Boq3} may be undertaken by pursuing \emph{mutatis mutandis} the proof of Theorem \ref{Boq2}, if we allow the additional assumption  that $T$ is invertible. In this case we apply the following result in lieu of Theorem \ref{Theelun}.
\begin{theorem}
Let $T$ be an ergodic invertible measure-preserving transformation of a complete probability space $(X,\mathcal{F},\mu)$, let $\mathcal{A} \colon X\times \mathbb{N} \to \Mat_d(\RR)$ be a measurable cocycle such that $\int \log^+\|\mathcal{A}(x,1)\|\,d\mu(x)<\infty$, and define the quantity $\lambda := \inf_{n \geq 1} \frac{1}{n} \int \log \|\mathcal{A}(x,n)\|\,d\mu(x)$. Then there exists a measurable $T$-invariant set $\Lambda \subseteq X$ satisfying $\mu(\Lambda)=1$ with the following properties. There exists an integer $p \in \{1,\ldots,d\}$, a finite sequence $\lambda=\lambda_1>\ldots>\lambda_p\geq -\infty$, and a corresponding sequence of measurable functions $F_1,\ldots,F_p$ from $\Lambda$ into the Grassmannian of $\RR^d$, such that for every $x \in \Lambda$ we have $F_1(x) \oplus \cdots \oplus F_p(x) = \RR^d$, $\mathcal{A}(x,n)F_i(x) \subseteq F_i(T^nx)$ for each $1 \leq i \leq p$ and $n \in \NN$, and
\[\lim_{n \to \infty}\frac{1}{n} \log \frac{\|\mathcal{A}(x,n)v\|}{\|v\|}= \lambda_i\]
uniformly for $v \in F_i(x) \setminus \{0\}$.
\end{theorem}
For a proof see \cite{FLQ}. The part of the statement dealing with uniform convergence is not declared in a completely explicit fashion in that article but features clearly in the proof. Since $\Mat_d(\RR)$ is separable, a suitable analogue of Lemma \ref{lemmer} may be proved easily without the additional requirement of $\mu$-continuity.

To complete the proof of Theorem \ref{Boq3} it remains to show that its result may be extended from the case of invertible $T$ to the general case. 
Given a measure-preserving transformation $T$ of a probability space $(X,\mathcal{F},\mu)$, recall from e.g. \cite{CFS} that there exist an invertible transformation $\hat T$ of a probability space $(\hat X,\hat{\mathcal{F}},\hat \mu)$ and a measurable map $\pi \colon \hat X \to X$ such that $\pi^*\hat\mu = \mu$ and $T \circ \pi = \pi \circ \hat T$ $\hat\mu$-a.e.
Now, given a cocycle $\mathcal{A} \colon X \times \mathbb{N} \to \Mat_d(\RR)$ which satisfies the conditions of Theorem \ref{Boq3}, note that the function $\hat {\mathcal{A}} \colon \hat X \times \mathbb{N} \to \Mat_d(\RR)$ defined by $\hat {\mathcal{A}}(x,n) := \mathcal{A}(\pi x,n)$ is a measurable cocycle with respect to $\hat T$. Since $\int_X \log^+ \|\mathcal{A}(x,1)\|d\mu = \int_{\hat X} \log^+\|\hat{\mathcal{A}}(x,1)\|d\hat\mu(x)$ by construction, $\hat{\mathcal{A}}$ meets the desired integrability condition and we obtain
\[\mu\left(\left\{x \in X \colon \limsup_{n\to\infty}\rho(\mathcal{A}(x,n))^{1/n} = \lim_{n \to \infty} \|\mathcal{A}(x,n)\|^{1/n}\right\}\right)\]
\[=\hat\mu\left(\left\{x \in \hat X \colon \limsup_{n\to\infty}\rho(\hat{\mathcal{A}}(x,n))^{1/n} = \lim_{n \to \infty} \|\hat{\mathcal{A}}( x,n)\|^{1/n}\right\}\right)=1\]
by applying Theorem \ref{Boq3} in the invertible case.

\section{Proof of Theorem \ref{Boq}}
To begin the proof, we claim that it is sufficient to demonstrate Theorem \ref{Boq} under the additional hypotheses that $\mathsf{A}$ is compact and consists solely of injective elements of $\mathcal{B}(\XX)$. We shall first show that if Theorem \ref{Boq} holds for compact sets of bounded operators then it also must hold for \emph{precompact} sets of bounded operators. To see this, fix a precompact set $\mathsf{A} \subseteq \BofX$ and suppose that $\hat\varrho(\overline{\mathsf{A}})=\max\{\varrho_\chi(\overline{\mathsf{A}}),\varrho_r(\overline{\mathsf{A}})\}$ and $\varrho_f(\overline{\mathsf{A}})=\varrho_\chi(\overline{\mathsf{A}})$. It follows from Proposition \ref{basic} that the maps $\|\cdot\|_\chi,  \|\cdot\|_f \colon \BofX \to \RR$ are continuous, whereupon a simple inspection of the definitions yields $\hat\varrho(\mathsf{A})=\hat\varrho(\overline{\mathsf{A}})$, $\varrho_\chi(\mathsf{A}) = \varrho_\chi(\overline{\mathsf{A}})$ and $\varrho_f(\mathsf{A})=\varrho_f(\overline{\mathsf{A}})$ so that in particular \eqref{HE} holds. If $\hat\varrho(\overline{\mathsf{A}})=\varrho_\chi(\overline{\mathsf{A}})$ then \eqref{GBWF} is clearly satisfied and the argument is complete. If otherwise, given any small enough $\varepsilon>0$ there exist infinitely many $n\in\NN$ such that
\[\sup_{A \in \overline{\mathsf{A}}^n} \rho(A)^{1/n} > \hat\varrho(\overline{\mathsf{A}}) - \varepsilon >\sup_{A \in \overline{\mathsf{A}}^n} \|A\|_\chi^{1/n}.\]
Given such an $n$, choose any $B \in \overline{\mathsf{A}}^n$ with $\rho(B)^{1/n} >\hat\varrho(\overline{\mathsf{A}}) -\varepsilon$. Since $\rho(B)>\|B\|_\chi$ the operator $B$ has essential spectral radius strictly smaller than its spectral radius, and it follows easily that $B$ is a point of continuity of the spectral radius functional $\rho \colon \BofX \to \RR$ (see  \cite[Lemma 9.3]{ST1} for details). Consequently we have
\[\sup_{A \in \mathsf{A}^n} \rho(A)^{1/n} \geq \rho(B)^{1/n} > \hat\varrho(\overline{\mathsf{A}})-\varepsilon = \hat\varrho(\mathsf{A})-\varepsilon,\]
and since this holds for infinitely many $n$ we deduce that $\varrho_r(\mathsf{A}) \geq \hat\varrho(\mathsf{A})-\varepsilon$. We conclude that $\varrho_r(\mathsf{A})=\hat\varrho(\mathsf{A})$ and \eqref{GBWF} is satisfied.

We next show that if Theorem \ref{Boq} holds for compact sets of injective bounded operators then it must hold for all compact sets of bounded operators. We apply a trick used by R. Ma\~n\'e \cite{Mane} and P. Thieullen \cite{T} which resembles the construction of the invertible natural extension of a dynamical system. Define a new Banach space $(\XX_\infty,\|\cdot\|)$ by $\XX_\infty:= \XX^{\mathbb{N}}$ and $\|(v_i)_{i \in \mathbb{N}}\| = \sup\{\|v_i\| \colon i\in \mathbb{N}\}$. Let $(\alpha_i)_{i\in\mathbb{N}}$ be a strictly decreasing sequence in $(0,1]$ with the property that for any subadditive sequence $(a_n)_{n \in \mathbb{N}}$,
\[\lim_{n \to \infty} \frac{a_n}{n} = \lim_{n \to \infty}\frac{1}{n}\max_{0 \leq k \leq n} \left( a_{n-k}+\sum_{i=0}^k\log \alpha_i\right).\]
(The existence of such a sequence was proved in \cite{T}). Define a map $\mathfrak{E} \colon \mathcal{B}(\XX) \to \mathcal{B}(\XX_\infty)$ by
$\mathfrak{E}(L) v_1 = Lv_1$, $\mathfrak{E}(L)v_{i+1}=\alpha_i v_i$. Clearly $\mathfrak{E}(L)$ is an injective operator for any $L \in \mathcal{B}(\XX)$. The reader may easily verify that for each $n \in \mathbb{N}$ and $L_1,\ldots,L_n \in \mathcal{B}(\XX)$,
\[\|\mathfrak{E}(L_n)\cdots\mathfrak{E}(L_1)\|= \max_{0 \leq k \leq n}\left(\| L_{n-k}\cdots L_1\|.\prod_{i=0}^{k} \alpha_i\right)\]
with the same relation holding for the seminorms $\|\cdot\|_\chi$ and $\|\cdot\|_f$. As a particular consequence it follows that $\rho(\mathfrak{E}(L_n)\cdots \mathfrak{E}(L_1)) = \rho(L_n\cdots L_1)$ for any $L_1,\ldots,L_n$ and hence $\varrho_r(\mathfrak{E}(\mathsf{A})) = \varrho_r(\mathsf{A})$. Since for each $n \in \NN$
\[\sup_{A \in \mathfrak{E}(\mathsf{A})^n} \|A\|^{1/n} = \max_{0 \leq k \leq n} \left(\sup_{A \in \mathsf{A}^{n-k}} \|A\|.\prod_{i=0}^k \alpha_i\right)^{1/n}\]
and similarly for $\|\cdot\|_\chi$ and $\|\cdot\|_f$, we conclude that $\hat\varrho(\mathsf{A})=\hat\varrho(\mathfrak{E}(\mathsf{A}))$, $\varrho_\chi(\mathsf{A})=\varrho_\chi(\mathfrak{E}(\mathsf{A}))$ and $\varrho_f(\mathsf{A})=\varrho_f(\mathfrak{E}(\mathsf{A}))$. Since the map $\mathfrak{E}$ is clearly continuous, $\mathfrak{E}(\mathsf{A})$ is a compact subset of $\mathcal{B}(\XX_\infty)$, and so if the conclusion of Theorem \ref{Boq} is valid for the compact set of injective operators $\mathfrak{E}(\mathsf{A})$ then it must be valid for $\mathsf{A}$ also.

For the remainder of this section, therefore, we shall assume that $\mathsf{A} \subset \mathcal{B}(\XX)$ is a compact nonempty set of injective bounded operators. Without loss of generality we shall assume $\hat\varrho(\mathsf{A})>\varrho_\chi(\mathsf{A}) \geq 0$, since if this relation does not hold then $\hat\varrho(\mathsf{A})=\varrho_f(\mathsf{A})=\varrho_\chi(\mathsf{A})$ by Proposition \ref{basic}, and Theorem \ref{Boq} thus holds trivially.

The following theorem on subadditive function sequences derives from a theorem of S. J. Schreiber \cite{Sc}. A similar result was also given by R. Sturman and J. Stark independently of Schreiber's work \cite{StSt}. A complete proof may be found in \cite{M}.
\begin{theorem}\label{StarkSturman}
Let $T \colon X \to X$ be a continuous transformation of a compact metric space, and let $\mathcal{M}_T$ be the set of all $T$-invariant Borel probability measures on $X$. Let $(f_n)_{n \geq 1}$ be a sequence of upper semi-continuous functions $f_i \colon X \to \mathbb{R}\cup\{-\infty\}$ such that $f_{n+m}(x) \leq f_n(T^mx) + f_m(x)$ for every $x \in X$ and $n,m \in \NN$. Then
\[\lim_{n \to \infty}  \sup_{x \in X} \frac{1}{n}f_n(x) = \sup_{\mu \in \mathcal{M}_T} \inf_{n \geq 1} \frac{1}{n} \int f_n\,d\mu.\]
\end{theorem}

Define $X = \mathsf{A}^{\mathbb{Z}}$ and equip this set with the product topology under which it is compact and metrisable. Define a homeomorphism $T \colon X\to X$ by $T[(A_i)_{i \in \mathbb{Z}}]=(A_{i+1})_{i \in \mathbb{Z}}$ and a continuous map $\Pi \colon X\to \BofX$ by $\Pi[(A_i)_{i \in \mathbb{Z}}] = A_1$. Define a continuous cocycle $\mathcal{A} \colon X \times \NN \to \BofX$ by setting $\mathcal{A}(x)=\Pi(T^{n-1}x)\cdots \Pi(x)$ for each $x \in X$ and $n \geq 1$, and let $\mathcal{M}_T$ denote the set of all $T$-invariant Borel probability measures on $X$, which is nonempty by the Krylov-Bogolioubov Theorem. We have
\[\hat\varrho(\mathsf{A}) = \lim_{n \to \infty} \sup_{x \in X} \|\mathcal{A}(x,n)\|^{1/n},\qquad\varrho_\chi(\mathsf{A}) = \lim_{n \to \infty} \sup_{x \in X} \|\mathcal{A}(x,n)\|_\chi^{1/n},\]
\[\varrho_f(\mathsf{A}) = \lim_{n \to \infty} \sup_{x \in X} \|\mathcal{A}(x,n)\|_f^{1/n},\qquad\varrho_r(\mathsf{A}) = \limsup_{n \to \infty} \sup_{x \in X} \rho(\mathcal{A}(x,n))^{1/n}\]
directly from the definitions. Applying Theorem \ref{StarkSturman} with $f_n(x):=\log \|\mathcal{A}(x,n)\|$ we deduce
\[\log \hat\varrho(\mathsf{A}) = \sup_{\mu \in \mathcal{M}_T} \inf_{n \geq 1} \frac{1}{n} \int \log \|\mathcal{A}(x,n)\|\,d\mu(x).\]
By \cite[Lemma 3.5]{M} we may choose $\mu \in \mathcal{M}_T$ such that
\[\inf_{n\geq 1} \frac{1}{n} \int \log \|\mathcal{A}(x,n)\|\,d\mu(x)=\log \hat\varrho(\mathsf{A})> \log\varrho_\chi(\mathsf{A}).\]
Let $\Lambda$, $\lambda$, $\chi$ be as in Theorem \ref{Boq2}. By the subadditive ergodic theorem we have $\int \lambda \,d\mu = \log \hat\varrho(\mathsf{A})$ and thus there is a positive measure set $Z \subseteq \Lambda$ such that every $z \in Z$ satisfies
\[\lambda(z) \geq \log \hat\varrho(\mathsf{A}) > \log \varrho_\chi(\mathsf{A}) \geq \lim_{n \to \infty} \frac{1}{n} \log\|\mathcal{A}(z,n)\|_\chi =\chi(z).\]
Applying Theorem \ref{Boq2} we deduce that every $z \in Z$ satisfies
\[\limsup_{n \to \infty} \frac{1}{n}\log \rho(\mathcal{A}(z,n)) =\lambda(z) \geq\log \hat\varrho(\mathsf{A})\]
and thus $\log \varrho_r(\mathsf{A})\geq \log\hat\varrho(\mathsf{A})$. 
Since clearly $\varrho_r(\mathsf{A}) \leq \hat\varrho(\mathsf{A})$ we conclude that $\hat\varrho(\mathsf{A})=\varrho_r(\mathsf{A})$ which yields \eqref{GBWF}.

By Proposition \ref{basic} we have $\varrho_f(\mathsf{A}) \geq \varrho_\chi(\mathsf{A})$. Suppose $\varrho_f(\mathsf{A})>\varrho_\chi(\mathsf{A})$. Applying Theorem \ref{StarkSturman} with $f_n(x) := \log \|\mathcal{A}(x,n)\|_f$ we deduce that there exists $\mu \in \mathcal{M}_T$ for which
\[\inf_{n\geq 1} \frac{1}{n}\int \log \|\mathcal{A}(x,n)\|_f\,d\mu(x)>\varrho_\chi(\mathsf{A}).\]
Via Proposition \ref{basic} this implies
\[\inf_{n\geq 1} \frac{1}{n}\int \log \|\mathcal{A}(x,n)\|\,d\mu(x)> \inf_{n\geq 1} \frac{1}{n}\int \log \|\mathcal{A}(x,n)\|_\chi\,d\mu(x)\]
and so Theorem \ref{Boq2} is applicable. Applying the subadditive ergodic theorem again we deduce that there is a positive-measure set $Z \subseteq \Lambda$ such that every $z \in Z$ satisfies
\[\lim_{n \to \infty} \frac{1}{n} \log \|\mathcal{A}(z,n)\|_f > \varrho_\chi(\mathsf{A}) \geq \lim_{n \to \infty} \frac{1}{n} \log \|\mathcal{A}(z,n)\|_\chi = \chi(z)\]
contradicting \eqref{anzab}. We conclude that $\varrho_f(\mathsf{A}) = \varrho_\chi(\mathsf{A})$ and the proof is complete.
\section{Acknowledgements}
This research was supported by EPSRC grant EP/E020801/1. The author would like to thank V. S. Shulman and Yu. V. Turovski\u{\i} for helpful suggestions.
\bibliographystyle{amsplain}
\bibliography{BWF}
\end{document}